\documentclass[12pt]{article}
\usepackage{lmodern}
\usepackage{microtype}

\usepackage[english]{babel}
\usepackage[all]{xy}
\usepackage{fancyhdr}
\usepackage{setspace, amsmath, amsthm, amssymb, amsfonts, amscd, epic, graphicx, ulem, dsfont}
\usepackage{amsthm}
\usepackage[T1]{fontenc}
\usepackage{tikz,xcolor,pgfplots}
\usepackage{cases}
\newtheorem{theorem}{Theorem}

\makeatletter

\@addtoreset{equation}{section}
\makeatother

\title{Harmonic basis vector fields on surfaces}

\author{Abdelkader Zagane\footnote{University Mustapha Stambouli Mascara, Faculty of Exact Sciences, Mascara 29000, Algeria. Email: abdelkader.zaagane@univ-mascara.dz} and Ahmed Mohammed Cherif\footnote{University Mustapha Stambouli Mascara, Faculty of Exact Sciences, Mascara 29000, Algeria. Email:
a.mohammedcherif@univ-mascara.dz
}}

\date{}

\begin{document}
\maketitle
	
\begin{abstract}
In this paper, we study local parameterizations of surfaces in the Euclidean space
$\mathbb{R}^3$ whose coordinate vector fields are harmonic sections with respect to the
induced Riemannian metric. After introducing the notion of harmonic basis vector fields
on a surface, we derive necessary and sufficient conditions under which the coordinate
vector fields $\frac{\partial}{\partial u}$ and $\frac{\partial}{\partial v}$ are harmonic.
We then focus on two important classes of surfaces: surfaces of revolution and translation
surfaces. For these families, we obtain a complete classification of all parameterizations
admitting harmonic basis vector fields.
\end{abstract}


\section{Introduction}
The study of harmonic objects on Riemannian manifolds has been a fundamental topic in
differential geometry and geometric analysis for many decades. A smooth map $\varphi:(M,g)\longrightarrow(N,h)$ between two Riemannian manifolds has the energy on a compact domain $D$ of $M$, defined by
\begin{equation}\label{eq1.1}
    E(\varphi;D)=\frac{1}{2}\int_D|d\varphi|^2 v^g,
\end{equation}
where $v^g$ is the volume element on $(M,g)$, and $|d\varphi|$ is the Hilbert-Schmidt norm of the differential $d\varphi$.
If $\varphi$ is a critical point of the energy functional (\ref{eq1.1}), then the map is referred to as harmonic. We have
$\tau(\varphi)=\operatorname{Tr}_g\nabla d\varphi=0$ as the Euler Lagrange equation for (\ref{eq1.1}), where $\nabla d\varphi$ is the second fundamental form of the smooth map $\varphi$ (see \cite{BW,ES}).\\
 The systematic study of harmonic vector fields was initiated by Wood \cite{Wood} and
 later developed by several authors. Let $(TM, g_S)$ the tangent bundle of a Riemanian manifold $(M, g)$ equipped with the Sasaki metric $g_S$ defined by
 $$g_S(X^h,Y^h)=g_S(X^v,Y^v)=g(X,Y)\circ\pi,\quad g_S(X^h,Y^v)=0,$$
 for all $X,Y\in\Gamma(TM)$, where $\pi:TM\longrightarrow M$ is the natural projection,
 and $X^h$ (resp. $X^v$) the horizontal (resp. vertical) lift of a vector field $X$ on $M$ to $TM$ (see \cite{GK}).\\
 In \cite{G}, the author proves that the tension field of $X:(M,g)\longrightarrow(TM,g_S)$ is given by
 \begin{equation}\label{eq1.2**}
\tau(X)=(\operatorname{Tr}_gR(X,\nabla_\cdot X)\cdot)^h+(\operatorname{Tr}_g\nabla^2X)^v.
\end{equation}
Therfore, $X:(M,g)\longrightarrow(TM,g_S)$ is harmonic map if and only if
$$\operatorname{Tr}_gR(X,\nabla_\cdot X)\cdot=0,\quad\operatorname{Tr}_g\nabla^2X=0,$$
where $R$ is the Riemannian curvature of $(M,g)$. Important contributions include the investigation of harmonic sections of tangent
bundles \cite{GK}, the relationship between volume and energy of vector fields
\cite{G}, and the total bending of vector fields on Riemannian manifolds \cite{Wiegmink}.
Further results concerning existence and classification of harmonic vector fields can be found in \cite{LiuYu}.\\
A smooth vector field $X$ on $(M,g)$ is said to be a harmonic section if it is a critical point of the vertical energy
\begin{equation}\label{eq1.3}
    E^v(X;D)=\frac{1}{2}\int_D|\nabla X|^2 v^g.
\end{equation}
The corresponding Euler-Lagrange equation is given by $\overline{\Delta}X=\operatorname{Tr}_g\nabla^2X=0$.\\
Harmonic vector fields represent an important class characterized by their minimization of energy and their connection to both geometric analysis. In this work, we present a new classification of harmonic vector fields on surfaces. Specifically, we begin by introducing a class of parameterizations whose basis vector fields are harmonic when viewed as sections of the tangent bundle.\\
Let \( M^2 \) be a surface in the Euclidean space \( \mathbb{R}^3 \). We say that a local parameterization $$\phi: (u,v) \longmapsto (x(u,v), y(u,v), z(u,v)),$$ of \( M^2 \) has harmonic basis vector fields if \( \frac{\partial}{\partial u} \) and \( \frac{\partial}{\partial v} \) are harmonic sections on $(M^2,g)$, where $g$ is the induce Riemannian metric.\\
A fundamental theme in this area is the appearance of Liouville-type theorems,
which assert the nonexistence of nontrivial harmonic objects under suitable geometric
or analytic conditions. Recently, in \cite{cherif}, the author established several Liouville-type theorems for
harmonic and biharmonic maps. These results naturally motivate the study of analogous rigidity phenomena for harmonic
vector fields.\\
In this work, we focus on harmonic basis vector fields associated with local
coordinate systems on Riemannian surfaces. As a concrete and geometrically rich class of examples, we consider translation surfaces
and surfaces of revolution in $\mathbb{R}^3$. These surfaces admit natural coordinate systems, making them particularly suitable for the study of harmonic coordinate vector fields. By computing explicitly the rough Laplacian of the basis vector fields, we derive
necessary and sufficient conditions for their harmonicity. More precisely, we prove that a surface of revolution whose coordinate vector fields are harmonic must have zero Gaussian curvature. As a consequence, such surfaces are necessarily planes, cylinders, or cones. This classification result may be viewed as a Liouville-type theorem for harmonic basis
vector fields on surfaces of revolution, in the spirit of the rigidity results obtained
for harmonic maps in \cite{cherif}. In particular, we show that harmonic sections coincide with harmonic maps. As an application, we prove that for translation surfaces this condition characterizes either planar surfaces or a specific
nontrivial family described by explicit functions.

\section{Harmonic basis vector fields on translation surfaces}
A translation surface is a surface defined by the following parametric representation
\begin{align*}
\phi: I \times J \subset \mathbb{R}^2 &\longrightarrow \mathbb{R}^3, \\
(u,v) &\longmapsto \phi(u,v) = \alpha(u) + \beta(v)
\end{align*}
where $\alpha$ and $\beta$ are two smooth curves in $\mathbb{R}^3$, of class $C^{\infty}$, defined on open intervals $I$ and $J$, respectively. A well-known class of translation surfaces is generated by two planar curves that lie in orthogonal planes. This special case is referred to as a translation surface of plane type, and it takes the form
\begin{align}
\phi(u,v) = (u, v, f(u) + h(v)), \label{translation-surface}
\end{align}
where $f$ and $h$ are smooth real-valued functions defined on $I$ and $J$, respectively.
The basic vector fields associated with the translation surface translation-surface are given by
\begin{equation}\label{eq2.1}
\frac{\partial}{\partial u} = \phi_u = (1, 0, f'), \quad \frac{\partial}{\partial v} = \phi_v = (0, 1, h').
\end{equation}
The unit normal vector field of the surface is computed as
\begin{align}\label{eq2.2}
\mathbf{N} &= \frac{\phi_u \wedge \phi_v}{\|\phi_u \wedge \phi_v\|}
= \frac{1}{\sqrt{1 + f'^2 + h'^2}}\,(-f', -h', 1).
\end{align}
The coefficients of the first fundamental form are given by
\begin{align*}
E &= \langle \phi_u, \phi_u \rangle = 1 + f'^2, \\
F &= \langle \phi_u, \phi_v \rangle = f' h', \\
G &= \langle \phi_v, \phi_v \rangle = 1 + h'^2.
\end{align*}
Thus, the induced Riemannian metric takes the form
\begin{align}\label{eq2.3}
g &= E\,du^2 + 2F\,du\,dv + G\,dv^2 \nonumber\\
  &= (1 + f'^2)\,du^2 + 2f'h'\,du\,dv + (1 + h'^2)\,dv^2.
\end{align}
The second-order partial derivatives of the parametrization are
\begin{align*}
\phi_{uu} &= (0, 0, f''), \\
\phi_{uv} &= (0, 0, 0), \\
\phi_{vv} &= (0, 0, h'').
\end{align*}
The coefficients of the second fundamental form are then given by
\begin{align*}
L &= \langle \mathbf{N}, \phi_{uu} \rangle = \frac{f''}{\sqrt{1 + f'^2 + h'^2}}, \\
M &= \langle \mathbf{N}, \phi_{uv} \rangle = 0, \\
N &= \langle \mathbf{N}, \phi_{vv} \rangle = \frac{h''}{\sqrt{1 + f'^2 + h'^2}}.
\end{align*}
Hence, the second fundamental form is
\begin{equation}\label{eq2.4}
\mathrm{II} = \frac{f''}{\sqrt{1 + f'^2 + h'^2}}\,du^2 + \frac{h''}{\sqrt{1 + f'^2 + h'^2}}\,dv^2.
\end{equation}
The Gaussian curvature \( \mathbf{K} \) and the mean curvature \( \mathbf{H} \) of the surface are given by
\begin{align}
\mathbf{K} &= \frac{LN - M^2}{EG - F^2} = \frac{f'' h''}{(1 + f'^2 + h'^2)^2},\label{eq2.5} \\
\mathbf{H} &= \frac{GL + EN - 2FM}{2(EG - F^2)} = \frac{f''(1 + h'^2) + h''(1 + f'^2)}{2(1 + f'^2 + h'^2)^{\frac{3}{2}}}.\label{eq2.6}
\end{align}
The non-zero Christoffel symbols \( \Gamma_{ij}^k \) of the surface are given by
\begin{align*}
\Gamma_{11}^1 &= \frac{f' f''}{1 + f'^2 + h'^2}, \quad &
\Gamma_{11}^2 &= \frac{h' f''}{1 + f'^2 + h'^2}, \\
\Gamma_{22}^1 &= \frac{f' h''}{1 + f'^2 + h'^2}, \quad &
\Gamma_{22}^2 &= \frac{h' h''}{1 + f'^2 + h'^2}.
\end{align*}
We define the following orthonormal basic vector fields
$$e_1=\frac{1}{\sigma_2}\;\frac{\partial}{\partial u},\qquad
e_2=-\frac{f'h'}{\sigma_1 \sigma_2}\;\frac{\partial}{\partial u}+\frac{\sigma_2}{\sigma_1}\;\frac{\partial}{\partial v},$$
where $\sigma_1=\sigma_1(u,v)=\sqrt{1+f{'}^2+h{'}^2}$ and $\sigma_2=\sigma_2(u)=\sqrt{1+f{'}^2}$. By using the Christoffel symbols formulas, we have the following
\begin{equation}\label{eq2.7}
	\nabla_{e_1}e_1=-\frac{f'f''h{'}^2}{\sigma_1^2\sigma_2^4}\frac{\partial}{\partial u}+\frac{h'f''}{\sigma_1^2\sigma_2^2}\frac{\partial}{\partial v} ,\quad \nabla_{e_2}e_2=\frac{f'h{'}^2f''}{\sigma_1^2\sigma_2^4}\frac{\partial}{\partial u}.
\end{equation}

By explicitly computing the rough Laplacian of these vector fields with respect to the
Levi-Civita connection of $g$, we obtain a system of nonlinear ordinary differential
equations involving the functions $f$ and $h$. The following theorem provides necessary
and sufficient conditions for the harmonicity of the coordinate vector fields in terms
of this system.

\begin{theorem}
The basis vectors fields $\frac{\partial}{\partial u}$ and $\frac{\partial}{\partial v}$ given in (\ref{eq2.1}) are harmonic sections on $(M^2,g)$ if and only if
\begin{equation*}{\normalsize (S_1)\,\begin{cases}
f{''}^2\left(1+h{'}^2\right)\left(1+h{'}^2-f{'}^2\right)+f{'}f{'''}\left(1+h{'}^2\right)\left(1+h{'}^2+f{'}^2\right)
-f{'}^2f{''}h{''}\left(1+f{'}^2-h{'}^2\right)=0, \vspace{0.3cm}\\
h{'}\left[ f{'''}\left(1+h{'}^2\right)\left(1+f{'}^2+h{'}^2\right)-2f{'}f{''}\left(h{''}\left(1+f{'}^2\right)
+f{''}\left(1+h{'}^2\right)\right)\right]=0, \vspace{0.3cm}\\
f{'}\left[h{'''}\left(1+f{'}^2\right)\left(1+f{'}^2+h{'}^2\right)-2h{'}h{''}\left(f{''}\left(1+h{'}^2\right)
+h{''}\left(1+f{'}^2\right)\right)\right]=0,\vspace{0.3cm}\\
h{''}^2\left(1+f{'}^2\right)\left(1+f{'}^2-h{'}^2\right)+h{'}h{'''}\left(1+f{'}^2\right)\left(1+f{'}^2+h{'}^2\right)
-h{'}^2h{''}f{''}\left(1+h{'}^2-f{'}^2\right)=0.
\end{cases}}
\end{equation*}
\end{theorem}

\begin{proof}
First, we compute the Laplacian of $\frac{\partial}{\partial u}$. We have
\begin{equation}\label{eq2.8}
\Delta\left(\frac{\partial}{\partial u}\right)=\nabla_{e_1}\nabla_{e_1}\frac{\partial}{\partial u}-\nabla_{\nabla_{e_1}e_1}\frac{\partial}{\partial u}+\nabla_{e_2}\nabla_{e_2}\frac{\partial}{\partial u}-\nabla_{\nabla_{e_2}e_2}\frac{\partial}{\partial u}.
\end{equation}
We compute  $\nabla_{\nabla_{e_1}e_1}\frac{\partial}{\partial u}$ and $\nabla_{\nabla_{e_2}e_2}\frac{\partial}{\partial u}$.
By using the formulas (\ref{eq2.7}), we get
\begin{align*}
\nabla_{\nabla_{e_1}e_1}\frac{\partial}{\partial u}
&=\nabla_{\left(-\frac{f'f''h{'}^2}{\sigma_2^4\sigma_1^2}\frac{\partial}{\partial u}
+\frac{h'f''}{\sigma_1^2\sigma_2^2}\frac{\partial}{\partial v}\right)}\frac{\partial}{\partial u}\\
&=-\frac{f'f''h{'}^2}{\sigma_2^4\sigma_1^2} \nabla_{\frac{\partial}{\partial u}}\frac{\partial}{\partial u}
+\frac{h'f''}{\sigma_1^2\sigma_2^2}\nabla_{\frac{\partial}{\partial v}}\frac{\partial}{\partial u}\\
&=-\frac{f'f''h{'}^2}{\sigma_2^4\sigma_1^2}\left(\frac{f'f''}{\sigma_1^2}\frac{\partial}{\partial u}
+\frac{h'f''}{\sigma_1^2}\frac{\partial}{\partial v}\right)\\
&=-\frac{f{'}^2f{''}^2h{'}^2}{\sigma_1^4\sigma_2^4}\frac{\partial}{\partial u}
-\frac{f{'}f{''}^2h{'}^3}{\sigma_1^4\sigma_2^4}\frac{\partial}{\partial v},
\end{align*}
and the following		
\begin{align*}
\nabla_{\nabla_{e_2}e_2}\frac{\partial}{\partial u}&=\nabla_{\frac{f'f''h{'}^2}{\sigma_1^2\sigma_2^4}\frac{\partial}{\partial u}}\frac{\partial}{\partial u}\\
		&=\frac{f'f''h{'}^2}{\sigma_1^2\sigma_2^4}\nabla_{\frac{\partial}{\partial u}}\frac{\partial}{\partial u}\\
		&=\frac{f'f''h{'}^2}{\sigma_1^2\sigma_2^4}\left(\frac{f'f''}{\sigma_1^2}\frac{\partial}{\partial u}+\frac{h'f''}{\sigma_1^2} \frac{\partial}{\partial v}\right)\\
		&=\frac{f{'}^2f{''}^2h{'}^2}{\sigma_1^4\sigma_2^4}\frac{\partial}{\partial u}+\frac{f{'}f{''}^2h{'}^3}{\sigma_1^4\sigma_2^4}\frac{\partial}{\partial v}.
\end{align*}
Note that, $ \nabla_{\nabla_{e_2}e_2}\frac{\partial}{\partial u}=- \nabla_{\nabla_{e_1}e_1}\frac{\partial}{\partial u}$.
We compute 	$\nabla_{e_1}\nabla_{e_1}\frac{\partial}{\partial u}$ and $\nabla_{e_2}\nabla_{e_2}\frac{\partial}{\partial u}$. We have
	\begin{align*}
	\nabla_{e_1}\frac{\partial}{\partial u}&=	\nabla_{\frac{1}{\sigma_2}\frac{\partial}{\partial u}}\frac{\partial}{\partial u}\\
	&=\frac{1}{\sigma_2}\left(\frac{f'f''}{\sigma_1^2}\frac{\partial}{\partial u}+\frac{h'f''}{\sigma_1^2} \frac{\partial}{\partial v}\right)\\
	&=\frac{f'f''}{\sigma_1^2\sigma_2}\frac{\partial}{\partial u}+\frac{h'f''}{\sigma_1^2\sigma_2} \frac{\partial}{\partial v},
	\end{align*}
and the following
	\begin{align*}
	\nabla_{e_2}\frac{\partial}{\partial u}&=	\nabla_{-\frac{f'h'}{\sigma_1 \sigma_2}\;\frac{\partial}{\partial u}+\frac{\sigma_2}{\sigma_1}\;\frac{\partial}{\partial v}}\frac{\partial}{\partial u}\\
	&=-\frac{f'h'}{\sigma_1 \sigma_2}\nabla_{\frac{\partial}{\partial u}}\frac{\partial}{\partial u}+\frac{\sigma_2}{\sigma_1}\nabla_{\frac{\partial}{\partial v}}\frac{\partial}{\partial u}\\
	&=-\frac{f'h'}{\sigma_1 \sigma_2}\left(\frac{f'f''}{\sigma_1^2}\frac{\partial}{\partial u}+\frac{h'f''}{\sigma_1^2} \frac{\partial}{\partial v}\right)\\
	&=-\frac{f{'}^2f''h'}{\sigma_1^3\sigma_2}\frac{\partial}{\partial u}-\frac{f'f''h{'}^2}{\sigma_1^3\sigma_2}\frac{\partial}{\partial v}.
	\end{align*}
Therefore, the term $\nabla_{e_1}\nabla_{e_1}\frac{\partial}{\partial u}$ is given by
	\begin{align*}
\nabla_{e_1}\nabla_{e_1}\frac{\partial}{\partial u}&=\nabla_{\frac{1}{\sigma_2}\frac{\partial}{\partial u}}\left(\frac{f'f''}{\sigma_1^2\sigma_2}\frac{\partial}{\partial u}+\frac{h'f''}{\sigma_1^2\sigma_2} \frac{\partial}{\partial v}\right)\\
	&=\frac{1}{\sigma_2}\left(\frac{\partial}{\partial u}\left(\frac{f'f''}{\sigma_1^2\sigma_2}\right)\frac{\partial}{\partial u}+\frac{f'f''}{\sigma_1^2\sigma_2}\nabla_{\frac{\partial}{\partial u}}\frac{\partial}{\partial u}+\frac{\partial}{\partial u}\left(\frac{h'f''}{\sigma_1^2\sigma_2}\right)\frac{\partial}{\partial v}+\frac{h'f''}{\sigma_1^2\sigma_2}\nabla_{\frac{\partial}{\partial u}}\frac{\partial}{\partial v}\right)\\
	 &=\frac{1}{\sigma_2}\left(\frac{f{''}^2\sigma_1^2+f'f'''\sigma_1^2\sigma_2^2-2f{'}^2f{''}^2\sigma_2^2}{\sigma_1^4\sigma_2^3}\frac{\partial}{\partial u}+\frac{f'f''}{\sigma_1^2\sigma_2}\left(\frac{f'f''}{\sigma_1^2}\frac{\partial}{\partial u}+\frac{h'f''}{\sigma_1^2}\frac{\partial}{\partial v}\right)\right. \notag\\
	&\quad \left.\hspace{0.75cm}+\frac{f'''h'\sigma_1^2\sigma_2^2-f'f{''}^2h'\sigma_1^2-2f'f{''}^2h'\sigma_2^2}{\sigma_1^4\sigma_2^3}\frac{\partial}{\partial v}\right)\\ &=\frac{1}{\sigma_2}\left(\frac{f{''}^2\sigma_1^2+f'f'''\sigma_1^2\sigma_2^2-2f{'}^2f{''}^2\sigma_2^2}{\sigma_1^4\sigma_2^3}\frac{\partial}{\partial u}+\frac{f{'}^2f{''}^2\sigma_2^2}{\sigma_1^4\sigma_2^3}\frac{\partial}{\partial u}+\frac{f'f{''}^2h'\sigma_2^2}{\sigma_1^4\sigma_2^3}\frac{\partial}{\partial v}\right. \notag\\
	&\quad \left.\hspace{0.75cm}+\frac{f'''h'\sigma_1^2\sigma_2^2-f'f{''}^2h'\sigma_1^2-2f'f{''}^2h'\sigma_2^2}{\sigma_1^4\sigma_2^3}\frac{\partial}{\partial v}\right)\\
	&=\frac{f{''}^2\sigma_1^2+f'f'''\sigma_1^2\sigma_2^2-f{'}^2f{''}^2\sigma_2^2}{\sigma_1^4\sigma_2^4}\frac{\partial}{\partial u}+\frac{f'''h'\sigma_1^2\sigma_2^2-f'f{''}^2h'\sigma_1^2-f'f{''}^2h'\sigma_2^2}{\sigma_1^4\sigma_2^4}\frac{\partial}{\partial v},
	\end{align*}
and the term $\nabla_{e_2}\nabla_{e_2}\frac{\partial}{\partial u}$ is given by
	\begin{align*}
\nabla_{e_2}\nabla_{e_2}\frac{\partial}{\partial u}&=\nabla_{ -\frac{f'h'}{\sigma_1 \sigma_2}\;\frac{\partial}{\partial u}+\frac{\sigma_2}{\sigma_1}\;\frac{\partial}{\partial v}}\left(-\frac{f{'}^2f''h'}{\sigma_1^3\sigma_2}\frac{\partial}{\partial u}-\frac{f'f''h{'}^2}{\sigma_1^3\sigma_2}\frac{\partial}{\partial v}\right)\\
		&=-\frac{f'h'}{\sigma_1 \sigma_2}\nabla_{\frac{\partial}{\partial u}}\left(-\frac{f{'}^2f''h'}{\sigma_1^3\sigma_2}\frac{\partial}{\partial u}-\frac{f'f''h{'}^2}{\sigma_1^3\sigma_2}\frac{\partial}{\partial v}\right)\notag\\
		&\quad \hspace{0.25cm}+\frac{\sigma_2}{\sigma_1}\nabla_{\frac{\partial}{\partial v}}\left(-\frac{f{'}^2f''h'}{\sigma_1^3\sigma_2}\frac{\partial}{\partial u}-\frac{f'f''h{'}^2}{\sigma_1^3\sigma_2}\frac{\partial}{\partial v}\right)\\
		&=-\frac{f'h'}{\sigma_1 \sigma_2}\left(-\frac{\partial}{\partial u}\left(\frac{f{'}^2f''h'}{\sigma_1^3\sigma_2}\right)\frac{\partial}{\partial u}-\frac{f{'}^2f''h'}{\sigma_1^3\sigma_2}\nabla_{\frac{\partial}{\partial u}}\frac{\partial}{\partial u}-\frac{\partial}{\partial u}\left(\frac{f{'}f''h{'}^2}{\sigma_1^3\sigma_2}\right)\frac{\partial}{\partial v}\right)\notag\\
		&\quad \hspace{0.25cm}+\frac{\sigma_2}{\sigma_1}\left(-\frac{\partial}{\partial v}\left(\frac{f{'}^2f''h'}{\sigma_1^3\sigma_2}\right)\frac{\partial}{\partial u}-\frac{\partial}{\partial v}\left(\frac{f{'}f''h{'}^2}{\sigma_1^3\sigma_2}\right)\frac{\partial}{\partial v}-\frac{f{'}f''h{'}^2}{\sigma_1^3\sigma_2}\nabla_{\frac{\partial}{\partial v}}\frac{\partial}{\partial v}\right)\\
		&=-\frac{f'h'}{\sigma_1 \sigma_2}\left(\frac{-2f'h'f{''}^2\sigma_1^2\sigma_2^2-f{'}^2f'''h'\sigma_1^2\sigma_2^2+3f{'}^3f{''}^2h'\sigma_2^2+f{'}^3f{''}^2h'\sigma_1^2}{\sigma_1^5\sigma_2^3}\frac{\partial}{\partial u}\right.\notag\\
		&\quad \left.-\frac{f{'}^2f''h'}{\sigma_1^3\sigma_2}\left(\frac{f'f''}{\sigma_1^2}\frac{\partial}{\partial u}+\frac{f''h'}{\sigma_1^2}\frac{\partial}{\partial v}\right)\right.\notag\\
		&\quad \left.+\frac{-f{''}^2h{'}^2\sigma_1^2-f'f'''h{'}^2\sigma_1^2\sigma_2^2+3f{'}^2f{''}^2h{'}^2\sigma_2^2}{\sigma_1^5\sigma_2^3}\frac{\partial}{\partial v}\right)\notag\\
		&\quad \hspace{0.05cm}+\frac{\sigma_2}{\sigma_1}\left(\frac{-f{'}^2f''h''\sigma_1^2+3f{'}^2f''h{'}^2h''}{\sigma_1^5\sigma_2}\frac{\partial}{\partial u}+\frac{-2f'f''h'h''\sigma_1^2+3f'f''h{'}^3h''}{\sigma_1^5\sigma_2}\frac{\partial}{\partial v}\right.\notag\\
		&\quad \left.-\frac{f'f''h{'}^2}{\sigma_1^3\sigma_2}\left(\frac{f'h''}{\sigma_1^2}\frac{\partial}{\partial u}+\frac{h'h''}{\sigma_1^2}\frac{\partial}{\partial v}\right)\right)\\ &=\frac{1}{\sigma_1^6\sigma_2^4}\left(2f{'}^2f{''}^2h{'}^2\sigma_1^2\sigma_2^2+f{'}^3f'''h{'}^2\sigma_1^2\sigma_2^2-2f{'}^4f{''}^2h{'}^2\sigma_2^2-f{'}^4f{''}^2h{'}^2\sigma_1^2\right.\notag\\
		&\quad \left.\hspace{1.50cm}-f{'}^2f''h''\sigma_1^2\sigma_2^4+2f{'}^2f''h{'}^2h''\sigma_2^4\right)\frac{\partial}{\partial u}\notag\\
		&\quad +\frac{1}{\sigma_1^6\sigma_2^4} \left(f'f{''}^2h{'}^3\sigma_1^2+f{'}^2f'''h{'}^3\sigma_1^2\sigma_2^2-2f{'}^3f{''}^2h{'}^3\sigma_2^2-2f'f''h'h''\sigma_1^2\sigma_2^4\right.\notag\\
		&\quad \left.\hspace{1.90cm}+2f'f''h{'}^3h''\sigma_2^4\right)\frac{\partial}{\partial v}.
	\end{align*}
Replacing these formulas into (\ref{eq2.8}), we obtain
\begin{align*}
\Delta\left(\frac{\partial}{\partial u}\right)&=\nabla_{e_1}\nabla_{e_1}\frac{\partial}{\partial u}+\nabla_{e_2}\nabla_{e_2}\frac{\partial}{\partial u}\\
&=\frac{1}{\sigma_1^6\sigma_2^4}\left(f{''}^2\sigma_1^4+f'f'''\sigma_1^4\sigma_2^2-f{'}^2f{''}^2\sigma_1^2\sigma_2^2+2f{'}^2f{''}^2h{'}^2\sigma_1^2\sigma_2^2\right.\notag\\
&\quad \left.\hspace{1.40cm}+f{'}^3f'''h{'}^2\sigma_1^2\sigma_2^2-2f{'}^4f{''}^2h{'}^2\sigma_2^2-f{'}^4f{''}^2h{'}^2\sigma_1^2-f{'}^2f''h''\sigma_1^2\sigma_2^4\right.\notag\\
&\quad \left.\hspace{1.40cm}+2f{'}^2f''h{'}^2h''\sigma_2^4\right)\frac{\partial}{\partial u}
\notag\\
&\quad
+\frac{1}{\sigma_1^6\sigma_2^4}\left(f'''h'\sigma_1^4\sigma_2^2-f'f{''}^2h'\sigma_1^4-f'f{''}^2h'\sigma_1^2\sigma_2^2+f'f{''}^2h{'}^3\sigma_1^2\right.\notag\\
&\quad \left.\hspace{1.80cm}+f{'}^2f'''h{'}^3\sigma_1^2\sigma_2^2-2f{'}^3f{''}^2h{'}^3\sigma_2^2-2f'f''h'h''\sigma_1^2\sigma_2^4\right.\notag\\
&\quad \left.\hspace{1.80cm}+2f'f''h{'}^3h''\sigma_2^4\right)\frac{\partial}{\partial v}\\
&=\frac{1}{\sigma_1^6}\left(-f{'}^4f''h''+f{'}^3f'''h{'}^2+f{'}^3f'''+f{'}^2f''h{'}^2h''-f{'}^2f{''}^2-f{'}^2f{''}^2h{'}^2\right.\notag\\
&\quad \left.\hspace{0.9cm}-f{'}^2f''h''+f'f'''h{'}^4+2f'f'''h{'}^2+f'f'''+f{''}^2+2f{''}^2h{'}^2+f{''}^2h{'}^4\right)\frac{\partial}{\partial u}
\notag\\
&\quad+\frac{h'}{\sigma_1^6} \left(-2f{'}^3f''h''+f{'}^2f'''h{'}^2+f{'}^2f'''-2f'f{''}^2-2f'f{''}^2h{'}^2-2f'f''h''\right.\notag\\
&\quad \left.\hspace{1.3cm}+f'''h{'}^4+2f'''h{'}^2+f'''\right)\frac{\partial}{\partial v}\\
&=\frac{1}{\sigma_1^6} \left(f{''}^2\left(1+h{'}^2\right)\left(1+h{'}^2-f{'}^2\right)+f{'}f{'''}\left(1+h{'}^2\right)\left(1+h{'}^2+f{'}^2\right)\right.\notag\\
&\quad \left.\hspace{0.8cm}-f{'}^2f{''}h{''}\left(1-h{'}^2+f{'}^2\right)\right)\frac{\partial}{\partial u}\notag\\
&\quad +\frac{h{'}}{\sigma_1^6}\left(f{'''}\left(1+h{'}^2\right)\left(1+h{'}^2+f{'}^2\right)
-2f{'}f{''}\left[h{''}\left(1+f{'}^2\right)+f{''}\left(1+h{'}^2\right)\right]\right)\frac{\partial}{\partial v}.
\end{align*}
We compute now $\Delta\left(\frac{\partial}{\partial v}\right)$. We have
\begin{equation}\label{eq2.9}
\Delta\left(\frac{\partial}{\partial v}\right)=\nabla_{e_1}\nabla_{e_1}\frac{\partial}{\partial v}-\nabla_{\nabla_{e_1}e_1}\frac{\partial}{\partial v}+\nabla_{e_2}\nabla_{e_2}\frac{\partial}{\partial v}-\nabla_{\nabla_{e_2}e_2}\frac{\partial}{\partial v}.
\end{equation}
Calculate $\nabla_{\nabla_{e_1}e_1}\frac{\partial}{\partial v}$ and $\nabla_{\nabla_{e_2}e_2}\frac{\partial}{\partial v}$.
By using the formulas (\ref{eq2.7}), we find that
\begin{align*}
	\nabla_{\nabla_{e_1}e_1}\frac{\partial}{\partial v}&=\nabla_{\left(\frac{-f'f''h{'}^2}{\sigma_1^2\sigma_2^4}\frac{\partial}{\partial u}+\frac{f''h'}{\sigma_1^2\sigma_2^2}\frac{\partial}{\partial v}\right)}\frac{\partial}{\partial v}\\
	&=\frac{f''h'}{\sigma_1^2\sigma_2^2}\nabla_{\frac{\partial}{\partial v}}\frac{\partial}{\partial v}\\
	&=\frac{f'f''h'h''}{\sigma_1^4\sigma_2^2}\frac{\partial}{\partial u}+\frac{f''h{'}^2h''}{\sigma_1^4\sigma_2^2}\frac{\partial}{\partial v},
\end{align*}
and the following
\begin{align*}\nabla_{\nabla_{e_2}e_2}\frac{\partial}{\partial v}&=\nabla_{\frac{f'f''h{'}^2}{\sigma_1^2\sigma_2^4}\frac{\partial}{\partial u}}\frac{\partial}{\partial v}=0.
\end{align*}
We compute $\nabla_{e_1}\nabla_{e_1}\frac{\partial}{\partial v}$ and $\nabla_{e_2}\nabla_{e_2}\frac{\partial}{\partial v}$. We have
	\begin{align*}
	\nabla_{e_1}\frac{\partial}{\partial v}&=	\nabla_{\frac{1}{\sigma_2}\frac{\partial}{\partial u}}\frac{\partial}{\partial v}=0,
	\end{align*}
and the following
	\begin{align*}
	\nabla_{e_2}\frac{\partial}{\partial v}&=\nabla_{\left(-\frac{f'h'}{\sigma_1\sigma_2}\frac{\partial}{\partial u}+\frac{\sigma_2}{\sigma_1}\frac{\partial}{\partial v}\right)}\frac{\partial}{\partial v}\\
	&=\frac{\sigma_2}{\sigma_1}\nabla_{\frac{\partial}{\partial v}}\frac{\partial}{\partial v}\\
	&=\frac{f'h''\sigma_2}{\sigma_1^3}\frac{\partial}{\partial u}+\frac{h'h''\sigma_2}{\sigma_1^3}\frac{\partial}{\partial v}.
	\end{align*}
Therefore $\nabla_{e_1}\nabla_{e_1}\frac{\partial}{\partial v}=0$, and we compute
	\begin{align*}
		\nabla_{e_2}\nabla_{e_2}\frac{\partial}{\partial v}&=\nabla_{\left(-\frac{f'h'}{\sigma_1 \sigma_2}\;\frac{\partial}{\partial u}+\frac{\sigma_2}{\sigma_1}\;\frac{\partial}{\partial v}\right)}\left(\frac{f'h''\sigma_2}{\sigma_1^3}\frac{\partial}{\partial u}+\frac{h'h''\sigma_2}{\sigma_1^3}\frac{\partial}{\partial v}\right)\\
		&=-\frac{f'h'}{\sigma_1 \sigma_2}\nabla_{\frac{\partial}{\partial u}}\left(\frac{f'h''\sigma_2}{\sigma_1^3}\frac{\partial}{\partial u}+\frac{h'h''\sigma_2}{\sigma_1^3}\frac{\partial}{\partial v}\right)+\frac{\sigma_2}{\sigma_1}\nabla_{\frac{\partial}{\partial v}}\left(\frac{f'h''\sigma_2}{\sigma_1^3}\frac{\partial}{\partial u}+\frac{h'h''\sigma_2}{\sigma_1^3}\frac{\partial}{\partial v}\right)\\
		&=-\frac{f'h'}{\sigma_1 \sigma_2}\left(\frac{\partial}{\partial u}\left(\frac{f'h''\sigma_2}{\sigma_1^3}\right)\frac{\partial}{\partial u}+\frac{f'h''\sigma_2}{\sigma_1^3}\nabla_{\frac{\partial}{\partial u}}\frac{\partial}{\partial u}+\frac{\partial}{\partial u}\left(\frac{h'h''\sigma_2}{\sigma_1^3}\right)\frac{\partial}{\partial v}\right)\notag\\
		&\quad+\frac{\sigma_2}{\sigma_1}\left(\frac{\partial}{\partial v}\left(\frac{f'h''\sigma_2}{\sigma_1^3}\right)\frac{\partial}{\partial u}+\frac{\partial}{\partial v}\left(\frac{h'h''\sigma_2}{\sigma_1^3}\right)\frac{\partial}{\partial v}+\frac{h'h''\sigma_2}{\sigma_1^3}\nabla_{\frac{\partial}{\partial v}}\frac{\partial}{\partial v}\right)\\
		&=-\frac{f'h'}{\sigma_1 \sigma_2}\left(\frac{f''h''\sigma_1^2\sigma_2^2+f{'}^2f''h''\sigma_1^2-3f{'}^2f''h''\sigma_2^2}{\sigma_1^5\sigma_2}\frac{\partial}{\partial u}+\frac{f{'}^2f''h''\sigma_2}{\sigma_1^5}\frac{\partial}{\partial u}\right.\notag\\
		&\quad \left.\hspace{1.70cm}+\frac{f'f''h'h''\sigma_2}{\sigma_1^5}\frac{\partial}{\partial v}+\frac{f'f''h'h''\sigma_1^2-3f'f''h'h''\sigma_2^2}{\sigma_1^5\sigma_2}\frac{\partial}{\partial v}\right)\notag\\
		&\quad +\frac{\sigma_2}{\sigma_1}\left(\frac{f'h'''\sigma_1^2\sigma_2-3f'h'h{''}^2\sigma_2}{\sigma_1^5}\frac{\partial}{\partial u}+\frac{h{''}^2\sigma_1^2\sigma_2+h'h'''\sigma_1^2\sigma_2-3h{'}^2h{''}^2\sigma_2}{\sigma_1^5}\frac{\partial}{\partial v}\right.\notag\\
		&\quad \left.\hspace{1.40cm}+\frac{f'h'h{''}^2\sigma_2}{\sigma_1^5}\frac{\partial}{\partial u}+\frac{h{'}^2h{''}^2\sigma_2}{\sigma_1^5}\frac{\partial}{\partial v}\right)\\ &=\frac{1}{\sigma_1^6\sigma_2^2}\left(-f'f''h'h''\sigma_1^2\sigma_2^2-f{'}^3f''h'h''\sigma_1^2+2f{'}^3f''h'h''\sigma_2^2+f'h'''\sigma_1^2\sigma_2^4\right.\notag\\
		&\quad \left.\hspace{1.4cm}-2f'h'h{''}^2\sigma_2^4\right)\frac{\partial}{\partial u}\notag\\
		&\quad 	 +\frac{1}{\sigma_1^6\sigma_2^2}\left(2f{'}^2f''h{'}^2h''\sigma_2^2-f{'}^2f''h{'}^2h''\sigma_1^2+h{''}^2\sigma_1^2\sigma_2^4+h'h'''\sigma_1^2\sigma_2^4\right.\notag\\
		&\quad \left.\hspace{1.9cm}-2h{'}^2h{''}^2\sigma_2^4\right)	\frac{\partial}{\partial v}.
	\end{align*}
By replacing these formulas into (\ref{eq2.9}), we get
	\begin{align*}
	\Delta\left(\frac{\partial}{\partial v}\right)&=\nabla_{e_2}\nabla_{e_2}\frac{\partial}{\partial v}-\nabla_{\nabla_{e_1}e_1}\frac{\partial}{\partial v}\\ &=\frac{1}{\sigma_1^6\sigma_2^2}\left(-f'f''h'h''\sigma_1^2\sigma_2^2-f{'}^3f''h'h''\sigma_1^2+2f{'}^3f''h'h''\sigma_2^2+f'h'''\sigma_1^2\sigma_2^4\right.\notag\\
	&\quad \left.\hspace{1.4cm}-2f'h'h{''}^2\sigma_2^4-f'f''h'h''\sigma_1^2\right)\frac{\partial}{\partial u}\notag\\
	&\quad 	 +\frac{1}{\sigma_1^6\sigma_2^2}\left(2f{'}^2f''h{'}^2h''\sigma_2^2-f{'}^2f''h{'}^2h''\sigma_1^2+h{''}^2\sigma_1^2\sigma_2^4+h'h'''\sigma_1^2\sigma_2^4\right.\notag\\
	&\quad \left.\hspace{1.4cm}-2h{'}^2h{''}^2\sigma_2^4-f''h{'}^2h''\sigma_1^2\right)	\frac{\partial}{\partial v}.
	\end{align*}
Note that
	\begin{align*}
	f{'}^3f''h'h''\sigma_1^2+f'f''h'h''\sigma_1^2&=f'f''h'h''\sigma_1^2(1+f{'}^2)\\&=f'f''h'h''\sigma_1^2\sigma_2^2,\\
	f{'}^2f''h{'}^2h''\sigma_1^2+f''h{'}^2h''\sigma_1^2&=f''h{'}^2h''\sigma_1^2(1+f{'}^2)\\&=f''h{'}^2h''\sigma_1^2\sigma_2^2.
	\end{align*}
We conclude that
\begin{align*}
		\Delta\left(\frac{\partial}{\partial v}\right)&=\frac{1}{\sigma_1^6}
\left(-2f'f''h'h''\sigma_1^2+2f{'}^3f''h'h''+f'h'''\sigma_1^2\sigma_2^2-2f'h'h{''}^2\sigma_2^2\right)\frac{\partial}{\partial u}
\notag\\
&\quad+\frac{1}{\sigma_1^6}\left(2f{'}^2f''h{'}^2h''+h{''}^2\sigma_1^2\sigma_2^2+h'h'''\sigma_1^2\sigma_2^2-2h{'}^2h{''}^2\sigma_2^2-f''h{'}^2h''\sigma_1^2\right)\frac{\partial}{\partial v}\\
&=\frac{1}{\sigma_1^6}\left(-2f'f''h'h''-2f{'}^3f''h'h''-2f'f''h{'}^3h''+2f{'}^3f''h'h''+f'h'''+2f{'}^3h'''\right.\notag\\
&\quad \left.\hspace{0.9cm}+f'h{'}^2h'''+f{'}^5h'''+f{'}^3h{'}^2h'''-2f'h'h{''}^2-2f{'}^3h'h{''}^2\right)\frac{\partial}{\partial u}\notag\\
&\quad +\frac{1}{\sigma_1^6}\left(2f{'}^2f''h{'}^2h''+h{''}^2+2f{'}^2h{''}^2+h{'}^2h{''}^2+f{'}^4h{''}^2+f{'}^2h{'}^2h{''}^2+h'h'''\right.\notag\\
&\quad \left.\hspace{1.25cm}+2f{'}^2h'h'''+h{'}^3h'''+f{'}^4h'h'''+f{'}^2h{'}^3h'''-2h{'}^2h{''}^2-2f{'}^2h{'}^2h{''}^2\right.\notag\\
&\quad \left.\hspace{1.25cm}-f''h{'}^2h''-f{'}^2f''h{'}^2h''-f''h{'}^4h''\right)\frac{\partial}{\partial v},
\end{align*}
therefore
\begin{align*}
\Delta\left(\frac{\partial}{\partial v}\right)
&=\frac{f'}{\sigma_1^6}\left(-2f''h'h''-2f''h{'}^3h''+h'''+2f{'}^2h'''+h{'}^2h'''+f{'}^4h'''+f{'}^2h{'}^2h'''\right.\notag\\
&\quad \left.\hspace{0.9cm}-2h'h{''}^2-2f{'}^2h'h{''}^2\right)\frac{\partial}{\partial u}\notag\\
&\quad +\frac{1}{\sigma_1^6}\left(f{'}^2f''h{'}^2h''+h{''}^2+2f{'}^2h{''}^2-h{'}^2h{''}^2+f{'}^4h{''}^2-f{'}^2h{'}^2h{''}^2+h'h'''\right.\notag\\
&\quad \left.\hspace{1.3cm}+2f{'}^2h'h'''+h{'}^3h'''+f{'}^4h'h'''+f{'}^2h{'}^3h'''-f''h{'}^2h''-f''h{'}^4h''\right)\frac{\partial}{\partial v}\\
&=\frac{f{'}}{\sigma_1^6}\left(h{'''}\left(1+f{'}^2\right)\left(1+f{'}^2+h{'}^2\right)-2h{'}h{''}\left[f{''}\left(1+h{'}^2\right)+h{''}\left(1+f{'}^2\right)\right]\right)\frac{\partial}{\partial u}\notag\\
& \quad
+\frac{1}{\sigma_1^6} \left(h{''}^2\left(1+f{'}^2\right)\left(1+f{'}^2-h{'}^2\right)+h{'}h{'''}\left(1+f{'}^2\right)\left(1+f{'}^2+h{'}^2\right)\right.\notag\\
&\quad \left.\hspace{0.8cm}-h{'}^2h{''}f{''}\left(1-f{'}^2+h{'}^2\right)\right)\frac{\partial}{\partial v}.
\end{align*}
The proof is complete.
\end{proof}

As a consequence of the system $(S_1)$ obtained in the previous theorem, we are able to
completely classify translation surfaces whose coordinate vector fields are harmonic
sections.

\begin{theorem}\label{thm1}
The basis vectors fields $\frac{\partial}{\partial u}$ and $\frac{\partial}{\partial v}$ are harmonic sections on $(M^2,g)$ if and only if
\begin{enumerate}
		\item
    $\begin{cases}
	f(u)=a_0u+a_1\\
	h(v)=b_0v+b_1
	\end{cases}$ where $a_0,a_1,b_0,b_1\in \mathbb{R}$.\\\\
In this case the translation surface is plane, or\\
	\item $\begin{cases} f(u)=\pm \left[c_0\sqrt{\alpha\exp\left(\frac{2u}{c_0} \right)-1}-c_0\arctan\left(\sqrt{\alpha\exp\left(\frac{2 u}{c_0}\right)-1}\right)+c_1\right]\\h(v)=k_0 \end{cases}$\\
for some constants $c_0\in \mathbb{R}^*\;,\;k_0,c_1\in \mathbb{R}\;,\;\alpha \in \mathbb{R}^*_+$
such that $\alpha\exp\left(\frac{2u}{c_0} \right)-1>0$.
\end{enumerate}
\end{theorem}

\begin{proof}
Note that, the second equation of $(S_1)$ is equivalent to
$$h{'}=0\;\; \hbox{ or } \;\; f{'''}\left(1+h{'}^2\right)\left(1+f{'}^2+h{'}^2\right)-2f{'}f{''}\left(h{''}\left(1+f{'}^2\right)+f{''}\left(1+h{'}^2\right)\right)=0.$$
\textbf{Case 1:} If $h{'}=0$, i.e., $h(v)=k_0$ for some constant $k_0\in \mathbb{R}$. Implies
\begin{equation*}
 \begin{cases}
\Delta\left(\frac{\partial}{\partial u}\right)=\left(f{''}^2\left(1-f{'}^2\right)+f{'}f{'''}\left(1+f{'}^2\right)\right)\frac{\partial}{\partial u}, \vspace{0.3cm}\\
\Delta\left(\frac{\partial}{\partial v}\right)=0.
\end{cases}
\end{equation*}
Thus, the basic vector fields are harmonic if and only if
\begin{equation}
f{''}^2\left(1-f{'}^2\right)+f{'}f{'''}\left(1+f{'}^2\right)=0.\label{*}
\end{equation}
\textbf{Case 1.1:} If  $f{'}=0$ or $f{''}=0$, i.e., $f(u)=k_1\,u+k_2$ where $k_1,k_2\in \mathbb{R}$, we get
$$ \begin{cases}
\Delta\left(\frac{\partial}{\partial u}\right)=0,\\
\Delta\left(\frac{\partial}{\partial v}\right)=0.	
\end{cases}$$
\textbf{Case 1.2:} If $f{'}\neq 0$ and $f{''}\neq 0$, solving equation (\ref{*}), we find that
$$ f(u)=\pm \left[c_0\sqrt{\alpha\exp\left(\frac{2u}{c_0} \right)-1}
-c_0\,\arctan\left(\sqrt{\alpha\exp\left(\frac{2 u}{c_0}\right)-1}\right)+c_1\right],$$
for some constants $c_0\in \mathbb{R}^*$, $c_1\in \mathbb{R}$, and $\alpha \in \mathbb{R}^*_+$ such that $\alpha\exp\left(\frac{2u}{c_0} \right)-1>0$.\\
\textbf{Case 2:} If $h{'}\neq0$ and we assume that $$f{'''}\left(1+h{'}^2\right)\left(1+f{'}^2+h{'}^2\right)-2f{'}f{''}\left(h{''}\left(1+f{'}^2\right)+f{''}\left(1+h{'}^2\right)\right)=0,$$
That is, we have
\begin{equation}\label{a}
f{'''}\left(1+h{'}^2\right)\left(1+f{'}^2+h{'}^2\right)=2f{'}f{''}\left(h{''}\left(1+f{'}^2\right)+f{''}\left(1+h{'}^2\right)\right).
\end{equation}
By substituting equation (\ref{a}) into the first equation of $(S_1)$, we obtain
\begin{eqnarray}\label{b}
&&f{''}^2\left(1+h{'}^2\right)\left(1+h{'}^2-f{'}^2\right)+2f{'}^2f{''}\left(h{''}\left(1+f{'}^2\right)+f{''}\left(1+h{'}^2\right)\right)\\
&&\nonumber-f{'}^2h{''}f{''}\left(1+f{'}^2-h{'}^2\right)=0.
\end{eqnarray}
By simple simplification, the equation (\ref{b}) is equivalent to
\begin{align*}
\left(1+h{'}^2+f{'}^2\right)\left(f{''}^2\left(1+h{'}^2\right)+f{'}^2f{''}h{''}\right)=0.
\end{align*}
Since $1+h{'}^2+f{'}^2> 1$, we obtain
\begin{equation}\label{c}
f{''}\left[f{''}\left(1+h{'}^2\right)+f{'}^2h{''}\right]=0.
\end{equation}
\textbf{Case 2.1:} If $f{''}=0$ implies $f(u)=a_0 u+a_1$ with $a_0,a_1\in \mathbb{R}$, we get
\begin{align*}
\begin{cases}
\Delta\left(\frac{\partial}{\partial u}\right)=0,\vspace{0.3cm}\\
\Delta\left(\frac{\partial}{\partial v}\right)=f{'}\left[ h{'''}\left(1+f{'}^2\right)\left(1+f{'}^2+h{'}^2\right)
-2h{'}h{''}^2\left(1+f{'}^2\right)\right]\frac{\partial}{\partial u}\vspace{0.3cm}\\
\hspace{2cm}+\left[h{''}^2\left(1+f{'}^2\right)\left(1+f{'}^2-h{'}^2\right)+h{'}h{'''}\left(1+f{'}^2\right)\left(1+f{'}^2
+h{'}^2\right)\right]\frac{\partial}{\partial v}.
\end{cases}
\end{align*}
\textbf{Case 2.1.1:} If $f{'}=0$, i.e., $a_0=0$, we deduce
\begin{equation}
\begin{cases}
\Delta\left(\frac{\partial}{\partial u}\right)=0\vspace{0.3cm}\\
\Delta\left(\frac{\partial}{\partial v}\right)=0	
\end{cases} \;\hbox{ iff }\;\;\; h{''}^2\left(1-h{'}^2\right)+h{'}h{'''}\left(1+h{'}^2\right)=0.\label{d}
\end{equation}
By solving the equation (\ref{d}), we find that $h{''}= 0$, i.e., $h(v)=k_0\,v+k_1$ with $k_0\in \mathbb{R}^*$ and $k_1\in \mathbb{R}$, or
\begin{equation*}
h(v)=\pm \left[c_0\sqrt{\alpha\exp\left(\frac{2v}{c_0} \right)-1}-c_0\arctan\left(\sqrt{\alpha\exp\left(\frac{2 v}{c_0}\right)-1}\right)+c_1\right],
\end{equation*}
where $c_0\in \mathbb{R}^*$, $c_1\in \mathbb{R}$, and $\alpha \in \mathbb{R}^*_+$ with $\alpha\exp\left(\frac{2v}{c_0} \right)-1>0$.\\
\textbf{Case 2.1.2:} If $f{'}\neq 0$, i.e., $a_0\neq0$, the basic vector fields are harmonic if and only if
$$(S_2)\,\begin{cases}
 h{'''}\left(1+f{'}^2\right)\left(1+f{'}^2+h{'}^2\right)-2h{'}h{''}^2\left(1+f{'}^2\right)=0,\vspace{0.3cm}\\
h{''}^2\left(1+f{'}^2\right)\left(1+f{'}^2-h{'}^2\right)+h{'}h{'''}\left(1+f{'}^2\right)\left(1+f{'}^2+h{'}^2\right)=0.
\end{cases}$$
Hence,
\begin{align*}
(S_2) & \iff \begin{cases}
h{'''}\left(1+f{'}^2\right)\left(1+f{'}^2+h{'}^2\right)=2h{'}h{''}^2\left(1+f{'}^2\right)\\
h{''}^2\left(1+f{'}^2\right)\left(1+f{'}^2-h{'}^2\right)+h{'}h{'''}\left(1+f{'}^2\right)\left(1+f{'}^2+h{'}^2\right)=0
\end{cases}\\
& \iff \begin{cases}
h{'''}\left(1+f{'}^2\right)\left(1+f{'}^2+h{'}^2\right)=2h{'}h{''}^2\left(1+f{'}^2\right)\\
h{''}^2\left(1+f{'}^2\right)\left(1+f{'}^2-h{'}^2\right)+2h{'}^2h{''}^2\left(1+f{'}^2\right)=0 \end{cases}\\
& \iff  \begin{cases}
h{'''}\left(1+f{'}^2+h{'}^2\right)=2h{'}h{''}^2\\
h{''}^2\left(1+f{'}^2\right)+h{'}^2h{''}^2=0 \end{cases}\\
& \iff \begin{cases}
h{'''}\left(1+f{'}^2+h{'}^2\right)=2h{'}h{''}^2\\
h{''}^2\left(1+f{'}^2+h{'}^2\right)=0 \end{cases}
\end{align*}
Since $1+f{'}^2+h{'}^2>1$ thus,
$$(S_2) \iff \begin{cases}
h{'''}\left(1+f{'}^2+h{'}^2\right)=2h{'}h{''}^2\\
h{''}^2=0 \end{cases}$$
and therefore $h(v)=b_0 v+b_1\;\;with\;\;b_0\in \mathbb{R}^*$ and $b_1\in \mathbb{R}$.\\
\textbf{Case 2.2:} If $f{'}\neq 0,\;\;f{''}\neq 0,\;\;h{'}\neq 0,\; \;h{''}\neq 0$ and $f{''}\left(1+h{'}^2\right)+f{'}^2h{''}=0$.
We get $\frac{f{''}}{f{'}^2}=-\frac{h{''}}{1+h{'}^2}=a_0,$ where $a_0\in \mathbb{R}^*.$
So that
\begin{eqnarray*}
   f(u)&=&-\frac{1}{a_0}\,\ln| a_0\;u+a_1|+a_2,\;\;a_1,a_2\in \mathbb{R}\;and\;u\neq -\frac{a_1}{a_0}, \\
   h(v)&=&\frac{1}{a_0}\,\ln|\cos(a_0\;v+a_3)|+a_4,\;\;a_3,a_4\in \mathbb{R},\;a_0\,v+a_3\in \left]-\frac{\pi}{2},\frac{\pi}{2}\right[.
\end{eqnarray*}
 We have $f{'}^2=\frac{f{''}}{a_0},$ and $1+h{'}^2=-\frac{h{''}}{a_0}$ implies $f{'''}=2\,a_0\,f{'}f{''}$,
 by substituting this equations into the second equation of $(S_1)$, we obtain $$h{'}\left[2\,a_0\,f{'}f{''}\left(-\frac{{h{''}}}{a_0}\right)\left(-\frac{h{''}}{a_0}+\frac{f{''}}{a_0}\right)-2\,f{'}f{''}\left(h{''}\left(1+\frac{f{''}}{a_0}\right)+f{''}\left(-\frac{h{''}}{a_0}\right)\right)\right]=0,$$
 this is equivalently to $f{''}-h{''}=-a_0,$ hence
 $f{''}=k_1,\;h{''}=k_2,$ such that $k_1-k_2=-a_0$,
 thus $f=k_1\;u+c_1,\;\;h=k_2\;v+c_2,$ such that $k_1-k_2=-a_0,\;\;c_1,\;c_2\in \mathbb{R}$.
 In this case $f$ and $h$ are not solutions of the system $(S_1)$.
\end{proof}

Theorem \ref{thm1} illustrates how to obtain the surface mentioned in Figure \ref{fig1}.

\begin{figure}[h!]
\centering
\begin{tikzpicture}
\begin{axis}[
    view={120}{30},
    xlabel={$x$}, ylabel={$y$}, zlabel={$z$},
    domain=0.5:2.5,
    y domain=-2:2,
    samples=50,
    samples y=20,
    colormap={gray}{color=(gray!50) color=(gray!80)},
]
\addplot3[
    surf,
    fill=gray!60,   
    draw=black,     
]
({x}, {y}, {sqrt(exp(2*x)-1) - atan(sqrt(exp(2*x)-1))});
\end{axis}
\end{tikzpicture}
\caption{Example of a translation surface $z=\sqrt{e^{2x}-1}-\arctan(\sqrt{e^{2x}-1})$ whose basis vector fields are harmonic}\label{fig1}
\end{figure}
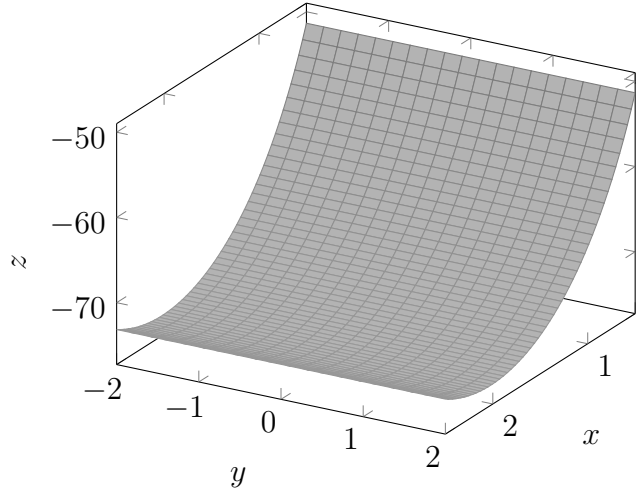

\section{Harmonic basis vector fields on surfaces of revolution}
A surface of revolution is a surface $M^2 \subset \mathbb{R}^3$ obtained by rotating a
planar curve $\Gamma$, called the generatrix, through a full revolution about a fixed
axis of rotation $\pi$, which in general does not intersect $\Gamma$ except possibly at
its endpoints.
Let $\Gamma$ be a curve defined by
\[
\Gamma : I \to \mathbb{R}^3, \qquad u \mapsto (f(u),0,h(u)),
\]
where $I$ is an open interval of $\mathbb{R}$ and $f,h$ are smooth functions on $I$ such
that $f'(u)\neq 0$ and $h'(u)\neq 0$ for all $u\in I$. By rotating the curve $\Gamma$ about
the $z$-axis, we obtain a surface of revolution $M^2$ parametrized by
\[
\phi : I \times (0,2\pi) \to \mathbb{R}^3, \qquad
(u,v) \mapsto \big(f(u)\sin v,\, f(u)\cos v,\, h(u)\big).
\]
The coordinate vector fields (basis vector fields) on the surface of revolution $M^2$
associated with the above parametrization are given by
\[
\frac{\partial}{\partial u}
= \big(f'(u)\sin v,\, f'(u)\cos v,\, h'(u)\big),
\qquad
\frac{\partial}{\partial v}
= \big(f(u)\cos v,\,-f(u)\sin v,\, 0\big).
\]
We set $\sigma(u)^2 = f'(u)^2 + h'(u)^2.$
Then the induced Riemannian metric $g$ on the surface $S$ is given, in the coordinates
$(u,v)$, by
\[
g =
\begin{pmatrix}
\sigma(u)^2 & 0 \\[0.2cm]
0 & f(u)^2
\end{pmatrix}.
\]
We obtain the rough Laplacians of the coordinate vector fields as
\begin{align}
\Delta \left(\frac{\partial}{\partial u}\right)
&= \left( \frac{\sigma''(u)\,\sigma(u) - \big(\sigma'(u)\big)^2}{\sigma(u)^4}
      - \frac{(f'(u))^2}{f(u)^2 \,\sigma(u)^2}
      + \frac{f'(u)\,\sigma'(u)}{f(u)\,\sigma(u)^3} \right)
      \frac{\partial}{\partial u}, \\
\Delta \left(\frac{\partial}{\partial v}\right)
&= \left( \frac{f''(u)\,\sigma(u) - f'(u)\,\sigma'(u)}{f(u)\,\sigma(u)^3} \right)
      \frac{\partial}{\partial v}.
\end{align}
The vector fields $\frac{\partial}{\partial u}$ and $\frac{\partial}{\partial v}$ on
the surface $M^2$ are harmonic if and only if
\begin{align}\label{1-1}
\begin{cases}
\dfrac{\sigma''(u)\,\sigma(u)-\big(\sigma'(u)\big)^{2}}{\sigma(u)^{4}}
 - \dfrac{(f'(u))^{2}}{f(u)^{2}\,\sigma(u)^{2}}
 + \dfrac{f'(u)\,\sigma'(u)}{f(u)\,\sigma(u)^{3}} = 0, \\[0.2cm]
\dfrac{f''(u)\,\sigma(u) - f'(u)\,\sigma'(u)}{f(u)\,\sigma(u)^{3}} = 0.
\end{cases}
\end{align}
Equivalently, multiplying through by appropriate powers of $f(u)$ and $\sigma(u)$,
the above system can be rewritten as
\begin{align}
\label{1-2}
\big(\sigma''(u)\,\sigma(u)-\sigma'(u)^{2}\big) f(u)^{2}
- (f'(u))^{2}\, \sigma(u)^{2}
+ f(u)\, \sigma(u)\, f'(u)\, \sigma'(u) &= 0, \\[0.2cm]
\label{1-3}
f''(u)\,\sigma(u) - f'(u)\,\sigma'(u) &= 0.
\end{align}
Replacing
\[
\sigma'(u) = \frac{f'(u) f''(u) + h'(u) h''(u)}{\sqrt{f'(u)^2 + h'(u)^2}}
\]
into equation \eqref{1-3}, we obtain
\begin{align}\label{1-4}
h'(u) \big(f''(u)\,h'(u) - h''(u)\,f'(u)\big) = 0.
\end{align}

\begin{figure}[hh!]\label{fig2}
\centering
\begin{tikzpicture}
\begin{axis}[
    view={120}{30},
    xlabel={$x$}, ylabel={$y$}, zlabel={$z$},
    domain=0:100,
    y domain=0:360,
    samples=50,
    samples y=20,
    colormap={gray}{color=(gray!50) color=(gray!80)},
]
\addplot3[
    surf,
    fill=gray!60,   
    draw=black,     
]
({1+sqrt(exp(2*x)-1)*cos(y)}, {sqrt(exp(2*x)-1)*sin(y)}, {sqrt(exp(2*x)-1)});
\end{axis}
\end{tikzpicture}
\caption{Example of a revolution surface whose basis vector fields are harmonic with $f(x)=h(x)=\sqrt{e^{2x}-1}$.}\label{fig2}
\end{figure}
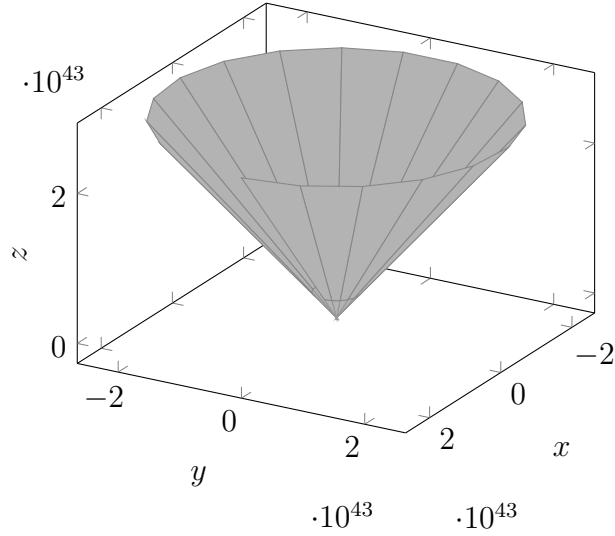

Recall that the Gauss curvature $K$ and the mean curvature $H$ of a surface of revolution
are given by
\begin{align*}
K &= \frac{h'(u)\,\big(f''(u)\,h'(u) - h''(u)\,f'(u)\big)}{f(u)\,\big(f'(u)^2 + h'(u)^2\big)^2}, \\
H &= \frac{-h'(u) + f(u)\,\big(h''(u)\,f'(u) - h'(u)\,f''(u)\big)}{2 f(u)}.
\end{align*}
Then equation \eqref{1-4} shows that the vector field $\frac{\partial}{\partial v}$ is
harmonic if and only if $K = 0$. From $K=0$, we obtain
\[
h'(u) = 0 \quad \text{or} \quad f''(u)\,h'(u) - h''(u)\,f'(u) = 0.
\]
In these cases, the mean curvature $H$ satisfies
\[
H = 0 \quad \text{or} \quad H = \frac{-h'(u)}{2 f(u)}.
\]
This implies that surfaces of revolution whose basis vector fields are harmonic necessarily have zero Gaussian curvature; in other words, such surfaces are either planar or parabolic.\\

We illustrate how to obtain the surface mentioned above in Figure \ref{fig2}.

\begin{theorem}\label{thm2}
A surface of revolution admits harmonic coordinate vector fields
if and only if it is a plane, a cylinder, or a cone.
\end{theorem}

\begin{proof}
We recall that $$\sigma(u)^2 = f'(u)^2 + h'(u)^2.$$
\textbf{Case 1:} Suppose $\sigma = c$, where $c$ is a positive real constant. Then
\[
h'(u) = \pm \sqrt{c^2 - f'(u)^2}, \quad \text{with } f'(u)^2 \leq c^2.
\]
In this case, formula (\ref{1-1}) becomes
\begin{align*}
\begin{cases}
\Delta\Big(\dfrac{\partial}{\partial u}\Big) = \dfrac{-f'(u)^2}{c^2 f(u)^2} \dfrac{\partial}{\partial u},\\[2mm]
\Delta\Big(\dfrac{\partial}{\partial v}\Big) = \dfrac{f''(u)}{c^2 f(u)} \dfrac{\partial}{\partial v}.
\end{cases}
\end{align*}
In the case $f'(u) = 0$, i.e., $f(u) = c \neq 0$, we have
\[
h'(u) = \sqrt{c^2 - a^2}, \quad \text{so} \quad h(u) = u \sqrt{c^2 - a^2}.
\]
Then, the curve is parametrized by
\[
(c, 0, u \sqrt{c^2 - a^2}), \quad \text{with } a, c \in \mathbb{R}_+^* \text{ and } a^2 < c^2.
\]
Consequently, the surface generated is parametrized as
\[
\phi(u,v) = \big(a \cos v, \; a \sin v, \; u \sqrt{c^2 - a^2}\big),
\]
which is a cylinder of revolution defined by the equation
\[
x^2 + y^2 = a^2.
\]
\textbf{Case 2:}
If $h'(u) = 0$, then $h$ is constant. Let us take $h(u) = a$, with $a \in \mathbb{R}_+$.
In this case, we have $\sigma(u) = \pm f'(u)$.
The vector field $\frac{\partial}{\partial v}$ is harmonic and
\[
\Delta\Big(\frac{\partial}{\partial u}\Big) =
\frac{f^2 f' f''' - f^2 (f'')^2 - (f')^4 + f (f')^2 f''}{f^2 (f')^4} \, \frac{\partial}{\partial u}.
\]
Thus, the vector field $\frac{\partial}{\partial u}$ is harmonic if and only if
\begin{equation}\label{2-1}
f^2 f' f''' - f^2 (f'')^2 - (f')^4 + f (f')^2 f'' = 0.
\end{equation}
Thus, we have
\[
f(u) = \sqrt{k e^{\alpha u} + \beta}, \quad
k \in \mathbb{R}^*, \;\; \beta \in \mathbb{R}_+, \;\; \text{with } \beta + k e^{\alpha u} > 0.
\]
In this case, the surface is parametrized by
\[
\phi(u,v)
= \big(\sqrt{k e^{\alpha u} + \beta} \cos v, \; \sqrt{k e^{\alpha u} + \beta} \sin v, \; a\big),
\]
where
$
a \in \mathbb{R}_+, \;\; k \in \mathbb{R}^*, \;\; \beta \in \mathbb{R}_+, \;\; \text{with } \beta + k e^{\alpha u} > 0.
$
Therefore, the surface of revolution is the plane defined by the equation
$
z = a.
$\\
\textbf{Case 3:}
If $h' \neq 0$ and $f'' h' - h'' f' = 0$  ($K = 0$), then the surface is of parabolic type. We distinguish two subcases. \\
\textbf{Case 3.1:} If $f' = 0$. Suppose $f(u) = a \neq 0$, then $\sigma(u) = \pm h'(u)$. We have
\[
\Delta\Big(\frac{\partial}{\partial v}\Big) = 0,
\]
but
\[
\Delta\Big(\frac{\partial}{\partial u}\Big) = 0 \quad \text{if} \quad
\frac{\sigma'' \sigma - (\sigma')^2}{\sigma^4} - \frac{(f')^2}{f^2 \sigma^2} + \frac{f' \sigma'}{f \sigma^3} = 0.
\]
This is equivalent to
\[
\frac{(\sigma'' \sigma - (\sigma')^2) f^2 - (f')^2 \sigma^2 + f \sigma f f' \sigma}{f^2 \sigma^4} = 0
\quad \Longleftrightarrow \quad
\frac{h''' h' - (h'')^2}{(h')^2} = 0. \label{2-2}
\]
Therefore, the basis vector fields are harmonic if and only if
\begin{equation}
h''' h' - (h'')^2 = 0.
\end{equation}
We now discuss two subcases for the third case. \\

\textbf{Case 3.1.1:} If $h'' = 0$. In this situation, the condition
\[
h''' h' - (h'')^2 = 0
\]
is automatically satisfied, which implies
\[
\Delta\Big(\frac{\partial}{\partial u}\Big) = 0.
\]

Thus, the basis vector fields are harmonic if
\[
f(u) = a, \quad h(u) = b u + c, \quad \text{where } a, b \in \mathbb{R}^*, \; c \in \mathbb{R}.
\]

The corresponding surface is parametrized by
\[
\phi(u,v) = (a \cos v, \; a \sin v, \; b u + c), \quad
a, b \in \mathbb{R}_+^*, \; c \in \mathbb{R}.
\]

Hence, the surface is a cylinder of revolution defined by the equation
\[
x^2 + y^2 = a^2, \quad a \in \mathbb{R}_+^*.
\]

\textbf{Case 3.1.2:} If $h'' \neq 0$. From equation (\ref{2-2}), we have $\frac{h'''}{h''} = \frac{h''}{h'}$. So that
$h(u) = \frac{b}{a} e^{a u} + c,$ where $a, b \in \mathbb{R}^*, \; c \in \mathbb{R}.$
The basis vector fields are harmonic if
\[
f(u) = a, \quad h(u) = \frac{b}{a} e^{a u} + c, \quad a, b \in \mathbb{R}^*, \; c \in \mathbb{R}.
\]
Hence, the surface is parametrized by
\[
\phi(u,v) = \Big(a \cos v, \; a \sin v, \; \frac{b}{a} e^{a u} + c \Big), \quad
a, b \in \mathbb{R}_+^*, \; c \in \mathbb{R}.
\]

Moreover, the surface is a surface of revolution, namely a cylinder defined by
\[
x^{2}+y^{2}=a^{2}, \qquad a\in \mathbb{R}_{+}^{*}.
\]

\textbf{Case 3.2:} Assume that ${f'\neq 0}$. Then we have
\begin{align*}
f''h'-h''f' = 0
&\;\Longleftrightarrow\; \frac{f''}{f'}=\frac{h''}{h'} \\
&\;\Longleftrightarrow\; \ln |f'|=\ln |h'| \\
&\;\Longleftrightarrow\; f'=a\,h', \qquad a\in\mathbb{R}^{*} \\
&\;\Longleftrightarrow\; f(u)=a\,h(u)+b,
\qquad a\in\mathbb{R}^{*},\; b\in\mathbb{R}.
\end{align*}

Furthermore,
\begin{align*}
\sigma^{2}=f'^{2}+h'^{2}
&\;\Longrightarrow\; \sigma^{2}=(1+a^{2})\,f'^{2} \\
&\;\Longrightarrow\; \sigma=\pm\sqrt{1+a^{2}}\,f'.
\end{align*}

In this case, we obtain
\begin{align*}
\begin{cases}
\displaystyle
\Delta\!\left(\frac{\partial}{\partial u}\right)
=\left(
\frac{\sigma''\sigma-\sigma'^{2}}{\sigma^{4}}
-\frac{f'^{2}}{f^{2}\sigma^{2}}
+\frac{f'\sigma'}{f\sigma^{3}}
\right)\frac{\partial}{\partial u},
\\[1em]
\displaystyle
\Delta\!\left(\frac{\partial}{\partial v}\right)=0,
\end{cases}
\end{align*}
which is equivalent to
\begin{align*}
\begin{cases}
\displaystyle
\Delta\!\left(\frac{\partial}{\partial u}\right)
=\frac{(\sigma''\sigma-\sigma'^{2})f^{2}
-f'^{2}\sigma^{2}
+f\sigma f'\sigma'}
{f^{2}\sigma^{4}}
\frac{\partial}{\partial u},
\\[1em]
\displaystyle
\Delta\!\left(\frac{\partial}{\partial v}\right)=0,
\end{cases}
\end{align*}
and hence
\begin{align*}
\begin{cases}
\displaystyle
\Delta\!\left(\frac{\partial}{\partial u}\right)
=\frac{1}{a}
\frac{
h^{2}h'h'''-h^{2}h''^{2}-h'^{4}+hh'^{2}h''
}{
h^{2}h'^{4}
}
\frac{\partial}{\partial u},
\\[1em]
\displaystyle
\Delta\!\left(\frac{\partial}{\partial v}\right)=0.
\end{cases}
\end{align*}
We obtain that the coordinate vector fields are harmonic if
\[
f(u)=a\,h(u)+b, \qquad a\in\mathbb{R}^{*}, \; b\in\mathbb{R}.
\]

Moreover, the function $h$ satisfies equation~\eqref{2-1} with $h'\neq 0$, namely
\[
h^{2}h'h'''-h^{2}h''^{2}-h'^{4}+h h'^{2}h''=0.
\]
Solving this differential equation yields
\[
\begin{cases}
h(u)=\sqrt{k e^{\alpha u}+\beta},\\[0.5em]
f(u)=a\sqrt{k e^{\alpha u}+\beta}+b,
\end{cases}
\]
where
\[
a,k\in\mathbb{R}^{*}, \qquad b,\beta\in\mathbb{R},
\qquad \text{with } \beta+k e^{\alpha u}>0.
\]

Hence, the surface is parametrized by
\begin{align*}
\phi(u,v)
&=\bigl((a h(u)+b)\cos v,\,(a h(u)+b)\sin v,\,h(u)\bigr)\\
&=\Bigl(\bigl(a\sqrt{k e^{\alpha u}+\beta}+b\bigr)\cos v,\,
\bigl(a\sqrt{k e^{\alpha u}+\beta}+b\bigr)\sin v,\,
\sqrt{k e^{\alpha u}+\beta}\Bigr).
\end{align*}

Therefore, the surface is a cone parametrized by $\phi$. Its Gaussian curvature and mean curvature are given by
\[
K=0, \qquad
H=\frac{\alpha k e^{\alpha u}}{2\sqrt{k e^{\alpha u}+\beta}}\neq 0.
\]

\end{proof}

\subsection*{Conflict of interest} The authors declare no conflict of interest.

\subsection*{Data Availability} Not applicable.



\begin{thebibliography}{99}

\bibitem{BW}
P. Baird, J. C. Wood,
{\it Harmonic morphisms between Riemannian manifolds},
Clarendon Press, Oxford (2003).

\bibitem{Chen}
B.-Y. Chen,
{\it Geometry of submanifolds},
Marcel Dekker, New York (1973).


\bibitem{DoCarmo}
M. do Carmo,
{\it Differential Geometry of Curves and Surfaces},
Prentice-Hall, Englewood Cliffs, NJ (1976).

\bibitem{DillenFastenakels}
F. Dillen, J. Fastenakels,
{\it Constant angle surfaces in Euclidean spaces},
Houston J. Math. \textbf{38}, 1067--1088 (2012).

\bibitem{ES}
J. Eells, J. H. Sampson,
{\it Harmonic mappings of Riemannian manifolds},
Amer. J. Math. \textbf{86}, 109--160 (1964).

\bibitem{Filip}
R. Filipkiewicz,
{\it Four dimensional geometries},
PhD thesis, University of Warwick, Great Britain (1983).

\bibitem{G}
O. Gil-Medrano,
{\it Relationship between volume and energy of vector fields},
Differential Geom. Appl. \textbf{15}, 137--152 (2001).

\bibitem{GK}
S. Gudmundsson, E. Kappos,
{\it On the Geometry of Tangent Bundles},
Expo. Math. \textbf{20}, 1--41 (2002).

\bibitem{H1}
R. Hamilton,
{\it The Ricci flow on surfaces},
Contemp. Math. \textbf{71}, 237--262 (1988).

\bibitem{H2}
R. Hamilton,
{\it The Harnack estimate for the Ricci flow},
J. Differential Geom. \textbf{37}, 225--243 (1993).

\bibitem{LiuYu}
H. Liu, Y. Yu,
{\it Harmonic vector fields on Riemannian manifolds},
J. Math. Anal. Appl. \textbf{339}, 219--226 (2008).

\bibitem{cherif}
A. Mohammed Cherif,
{\it On the nonexistence of harmonic and bi-harmonic maps},
Bull. Math. Soc. Sci. Math. Roumanie, \textbf{116}, 173--184 (2025).

\bibitem{ON}
B. O'Neill,
{\it Semi-Riemannian Geometry},
Academic Press, New York (1983).

\bibitem{Onda}
K. Onda,
{\it Lorentz Ricci Solitons on $3$-dimensional Lie groups},
Geom. Dedicata \textbf{147}, 313--322 (2010).

\bibitem{Wiegmink}
G. Wiegmink,
{\it Total bending of vector fields on Riemannian manifolds},
Math. Ann. \textbf{303}, 325--344 (1995).

\bibitem{Wood}
C. M. Wood,
{\it Harmonic sections of vector bundles},
Proc. London Math. Soc. \textbf{33}, 129--151 (1976).

\end{thebibliography}
\end{document}